\documentclass[10pt,a4paper]{article}
\usepackage{amsmath,amsxtra,amssymb,latexsym,amscd,amsfonts,multicol,enumerate,ifthen,stmaryrd,indentfirst,amsthm,amstext}
\usepackage[mathscr]{eucal}
\usepackage{graphics, graphpap}
\usepackage[pdftex]{graphicx}
\usepackage{xspace}
\usepackage{array, tabularx, longtable}
\usepackage{multicol,color}
\usepackage{mathrsfs}
\usepackage{makeidx}
\usepackage{indentfirst}
\usepackage{subfigure}

\usepackage{hyperref}
\belowdisplayshortskip =\abovedisplayshortskip
\def\NoBlackBoxes{\overfullrule=0pt }
\NoBlackBoxes
\theoremstyle{plain}
\newtheorem{theorem}{Theorem}
\newtheorem{lemma}{Lemma}
\newtheorem{definition}{Definition}

\newtheorem{corollary}{Corollary}
\newcommand\R{{\mathbb R}}
\newcommand\E{{\textnormal E}}
\newcommand\Var{{\rm Var}}
\newcommand\PP{{\textnormal P}}
\title{ \bf The record method for two and three dimensional parameters random fields}
\author{ Jean-Marc Aza\"\i s and Viet-Hung Pham 
\\Institut de math\'{e}matiques de Toulouse \\ Universit\'{e} Paul Sabatier (Toulouse III) \\ 118, route de Narbonne \\ 31062 TOULOUSE Cedex 09}
\begin{document} 
\maketitle

\abstract
Let $S$ be a regular set of $\R^d$ and $X : S\rightarrow \R$ be Gaussian field with regular paths. In order to give bound to the tail of the distribution of the maximum, we use the record method of Mercadier. We present some new form in dimension 2 and extend it to dimension 3 using the result of the expectation of the absolute value of quadratic forms by Li and Wei. Comparison with other methods is conducted.

\textbf{Key-words:} Stochastic processes, Gaussian fields, Rice formula, distribution of the maximum.\\
\indent \textbf{Classifications:} 60G15, 60G60, 60G70.
\section{Introduction}
The problem of computing the tail of the maximum has a lot of applications in spatial statistics, image processing, oceanography, genetics etc ..., see for example Cressie and Wikle \cite{7}. It is exactly solved only for about ten processes with parameter of dimension 1, see Aza\"\i s and Wschebor \cite{5} p4 for a complete list. In the other cases, one has to use some approximations. Several methods have been used, in particular
\begin{itemize}
 \item The tube method, Sun \cite{17}.
 \item Double sum method, Piterbarg \cite{14}.
 \item Euler characteristic method see, for example, Adler and Taylor \cite{1}.
 \item Rice or direct method, Aza\"\i s and Delmas \cite{3}, Aza\"\i s and Wschebor \cite{5}.
 \end{itemize}
With respect to these methods, the record method which is the main subject of this paper and which is detailed in Section \ref{se2} has the advantage of simplicity and also the advantage of giving a bound  which is non asymptotic:  it is true for every level and not for large $u$ only.

 It has been introduced for one-parameter random processes by Rychlik \cite{16} and extended to two-parameter random fields by Mercadier \cite{13} to study the tail of the maximum of smooth Gaussian random fields on rather regular sets.\\ 
 It has two version, one is an exact implicit formula : Theorem 2 in \cite{13} that is interesting for numerical purpose and that  will not  be considered  here; the other form is a bound for the tail, see inequality (\ref{re1}) hereunder. 
 
 This bound has the advantage of  its simplicity. In particular it avoids  the computation of the expectation of the absolute value of the Hessian determinant as  in the direct method of \cite{4} but it works only dimension 2. \\
 For practical applications, the dimensions 2 and 3 (for the parameter set) are the most relevant so there is a need of an extension to dimension 3 and this is done in Section 3 using results on quadratic forms by Li and Wei \cite{11}.\\
The bound also has the drawback of demanding a parameterization of the boundary. For example, if we consider the version of Aza\"\i s and Wschebor (\cite{5}, Theorem 9.5 ) of the result of Mercadier, under some mild conditions on the set $S \subset \R^2$ and on the Gaussian process $X$, we have
\begin{align}\label{re1}
\PP \{M_S \geq u \} \leq & \PP\{Y(O)\geq u\} + \int_0^L \E(|Y'(l)|\mid Y(l)=u)p_{Y(l)}(u)\, dl \notag \\
& + \int_S \E(|X''_{11}(t)^-X'_2(t)^+|\mid X(t)=u,\, X'_1(t)=0) p_{X(t),X'_1(t)}(u,0)\, dt,
\end{align}
where
\begin{itemize}
\item $M_S $ is the maximum of $X(t)$ on the set $S$.
\item 
 $Y(l)=X(\rho (l))$ with $\rho: \; [0,L]\rightarrow \partial S$ is a parameterization of the boundary $\partial S $ by its length.
\item $\displaystyle X''_{ij}=\frac{\partial^2 X}{\partial x_i \partial x_j}$.
\item $p_Z(x)$: the value of the density function of random vector $Z$ at point $x$.
\item $x^+=\sup (x,0)$, \; $x^-=\sup (-x,0)$.
\end{itemize}
The proof is based on considering the point with minimal ordinate (second coordinate) on the level curve. As we will see, this point can be considered as a ``record point''. \\
So the second direction of generalizations is to propose nicer and stronger forms of the inequality (\ref{re1}). This is done in Section \ref{se2}. The result on quadratic form is presented in Section \ref{se4} and some numerical experiment is presented in Section \ref{se5}.

\subsection*{Notation}

\begin{itemize}
\item $S$ is some rather regular set included in $\R^2$ or $\R^3$.  $\partial S$ is its  boundary; $\overset{\circ}{S}$ is its interior.
\item $M_S=\underset{s\in S}{\max}\, X(s)$ where $X(s)$ is some rather regular process.
\item $\sigma_i$ is the surface measure of dimension $i$. It can be defined as a Hausdorff measure.
\item $X', \; X''$ are the first and second derivatives of the process $ X(t)$. In particular if $\alpha $ is some direction then $X'_{\alpha}$ is the derivative along the direction $\alpha$.
\item $M \preceq 0$ means that the square matrix $M$ is semi-definite negative.
\item $S^{ +\epsilon}$ is the tube around $S$, i.e
$$S^{+\epsilon}=\{s\in \R^2:\, \mbox{dist}(s,S)\leq \epsilon\}.$$
 \item $d_H$ is the Hausdorff distance between sets, defined by
$$d_H(S,T)=\inf \{\epsilon:\; S \subset T^{+\epsilon},\, T\subset S^{+\epsilon}\}.$$
\item $\varphi(x)$ and $\Phi(x)$ are the density and distribution function of a
 standard normal variable.  \\$\overline{\Phi}(x)=1-\Phi(x)$.
\end{itemize}

\section{The record method in dimension 2 revisited}\label{se2}
We will work essentially under the following assumption: \\
 \textbf{Assumption 1:}
$\{X(t),\; t \in NS \subset \R^2\}$ is a Gaussian stationary field, defined  in a neighborhood $NS$ of $S$ with $\mathcal{C}^1$ paths and such that there exists some direction, that will be assumed (without loss of generality)  to be the direction 
 of the first coordinate, in which the second derivative $X''_{11}(t)$ exists.\\
 We assume moreover the following normalizing conditions that  can always  be obtained by a scaling
$$
\E(X(t))=0, \; \Var(X(t))=1, \; \Var {X'(t)}=I_2 .
 $$
 Finally we assume that $ \Var(X''_{11}(t))>1$ which is true as soon as the spectral measure of the process restricted to the first axis is not concentrated on two opposite atoms. \\

 In some cases we will assume in addition \\
\textbf{Assumption 2:}
$X(t)$ is isotropic, i.e $\textnormal{Cov}(X(s),X(t))=\rho(\|t-s\|^2)$,  with $\mathcal{C}^2$ paths and $S$ is a convex polygon. \\
 Under Assumption 1 and 2 plus some light additional hypotheses, the Euler Characteritic (EC) method \cite{1} gives
$$
\PP\{M_S \geq u \} = \PP_E (u) +\mbox{Rest},
$$
 with 
 $$
 \PP_E (u) = 
 \overline{\Phi}(u)+\frac{\sigma_1(\partial S)}{2\sqrt{2\pi}} \varphi (u)+\frac{\sigma_2(S)}{2\pi}u\varphi (u),
 $$
where the rest is super exponentially small.\\
The direct method gives \cite{4}
\begin{align}
\PP\{M_S \geq u \} & \leq \PP_M (u) = \overline{\Phi}(u)\displaystyle +\frac{\sigma_1(\partial S)}{2\sqrt{2\pi}} \int_u^{\infty} \left[ c\varphi (x\textrm{/}c)+x\Phi (x\textrm{/c})\right]\varphi (x)dx \notag \\
&\displaystyle + \frac{\sigma_2(S)}{2\pi } \int_u^{\infty} \left[x^2-1+\frac{\displaystyle (8\rho''(0) )^{3/2}\exp(-x^2.(24\rho''(0) -2)^{-1})}{ \sqrt{24\rho''(0) -2}}\right]\varphi(x)dx, \label{bou}
\end{align}
where $c=\sqrt{\textnormal{Var} (X''_{11})-1}=\sqrt{12\rho''(0)-1}$.\\
The record method gives \cite{13}
$$\PP\{M_S \geq u \} \leq \overline{\Phi}(u)+\frac{\sigma_1(\partial S)}{\sqrt{2\pi}} \varphi (u)+\frac{\sigma_2(S)}{2\pi }\left[ c\varphi (u\textrm{/}c)+u\Phi (u\textrm{/c})\right]\varphi (u).
$$
 A careful examination of these equations shows that the main terms are almost the same except that in the record method the coefficient 
of $\sigma_1(\partial S)$ is twice too large. When $S$ is a rectangle $[0,T_1]\times[0,T_2]$, it is easy to prove that this coefficient 2 can be removed, see for example Exercise 9.2 in \cite{5}.
\medskip \\
The goal of this section is to extend the result above to more general sets and to fields satisfying Assumption 1 only.  The main result of this section is the following 
 %
\begin{theorem}\label{th1} Let $X$ satisfy the Assumption 1 and suppose that $S$ is the Hausdorff limit of connected polygons $S_n$. Then,
\begin{equation}\label{bou1}
\PP \{M_S \geq u\} \leq \overline{\Phi}(u)+ \frac{\liminf_n \sigma_1(\partial S_n) \varphi (u)}{2\sqrt{2\pi}}+\frac{\sigma_2(S)}{2\pi }\left[ c\varphi (u\textrm{/}c)+u\Phi (u\textrm{/c})\right]\varphi (u),
\end{equation}
where $c=\sqrt{\textnormal{Var} (X''_{11})-1}$.
\end{theorem}
{\bf Remark:} the choice of the direction of ordinates is arbitrary and is a consequence of the arbitrary choice of the the second derivative $X''_{11}$.
When the process $X(t)$ admits derivative in all direction, the choice that gives the sharpest bound consists in chosing as first axis, the direction $\alpha$ such that 
$ \Var(X''_{\alpha \alpha})$ is minimum.  \bigskip

Unfortunately the proof it is based on an exotic topological property of the set $S$ that will be called ''emptyable".
\begin{definition} The compact set $S$ is emptyable if there exists a point $O\in S$ which has minimal ordinate,
 and such that for every $s \in S$ there exists a continuous path inside $S$ from $O$ to $s$ with non decreasing ordinate.
\end{definition}
In other word, suppose that $S$ is filled with water and that gravity is in the usual direction; $S$ is emptyable if after making a small hole at $O$, all the water will empty out, see Figure \ref{fig1}. \\
\begin{figure}[!ht]
\begin{center}
\includegraphics[scale=0.65]{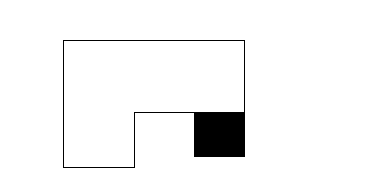}
\caption{Example of non-emptyable set. The non-emptyable part is displayed in black. } \label{fig1}
\end{center}
\end{figure}
\begin{proof} 
Step1 : Suppose for the moment that $X$ has $\mathcal{C}^{\infty}$ paths and that $S$ is an emptyable polygon. Considering the event $\{M_S \geq u\}$, we have
\begin{equation}
\label{pt1}
\PP \{M_S \geq u\} = \PP \{ X(O) \geq u\} + \PP \{ X(O) <u,M_S \geq u\}. 
\end{equation}
It is clear that if $X(O)< u$ and $M_S\geq u$, because $S$ is connected, the level curve
$$
\mathcal{C}_u=\{t \in S: \; X(t)=u\}
$$
is not empty, and there is at least one point $T$ on $\mathcal{C}_u$ with minimal ordinate. There are two possibilities: 
\begin{itemize}
\item $T$ is in the interior of $S$. In that case, suppose that there exists a point $s \in S$ with smaller ordinate than $T$ ($s_2<T_2$), such that $X(s)>u$. Then, due to the emptyable property, on the continuous path from $O$ to $s$ there would exist one point $s'$ with smaller ordinate than $T$, and with $X(s') =u$. 
 This is in contradiction with the definition of $T$. So we have proved that for every $s \in S$, $s_2<T_2$ we have $X(s) \leq u$. It is in the sense that $T$ can be considered as a record point. It implies that 
 $$
\{ X'_1(T)=0, \; X'_2(T) \geq 0, \; X''_{11}(T) \leq 0 \}. 
 $$
 The probability that there exists such a point is clearly bounded,  by the Markov inequality, by 
 $$
 \E\left(\textnormal{card}\{t\in S:\; X(t)=u, \; X'_1(t)=0, \; X'_2(t) \geq 0, \; X''_{11}(t) \leq 0\}\right). \notag \\
 $$
 Applying the Rice formula to the field $Z=(X,X'_1)$ from $\R^2$ to $\R^2$, we  get  that 
\begin{align}
&\quad \PP \{\exists \, t \in \overset{\circ}{S}: \, X(t)=u, \; t \; \mbox{has minimal ordinate on} \; \mathcal{C}_u \} \notag \\
\leq & \quad \displaystyle \int_S \E\left(|\det(Z'(t))| \mathbb{I}_{X'_2(t) \geq 0} \mathbb{I}_{X''_{11}(t) \leq 0}\mid Z(t)=(u,0)\right) \times p_{Z(t)}(u,0) \; dt \notag \\
= & \quad \displaystyle \sigma_2(S) \frac{\varphi(u)}{\sqrt{2\pi}}\; \E\left(X''^-_{11}(t)X^{\prime +}_2(t) \mid X(t)=u, X'_1(t)=0\right) \notag \\
= & \quad \displaystyle \sigma_2(S) \frac{\varphi(u)}{\sqrt{2\pi}} \; \E\left(X^{\prime +}_2(t)\right) \; \E\left(X''^-_{11}(t) \mid X(t)=u, X'_1(t)=0\right) \notag \\
= &\displaystyle \quad \sigma_2(S) \frac{\varphi (u)}{2\pi }\left[ c\varphi (u\textrm{/}c)+u\Phi (u\textrm{/}c)\right]. \label{pt2}
\end{align}
Note that the validity of the Rice formula holds true because the paths are of class $\cal{C}^\infty $ and that $X(t)$ and $X'_1(t)$ are independent. The computations above use some extra independences that are a consequence of the normalization of the process. The main point is that, under the conditioning 
$$ \det( Z'(t) ) = X''_{11}(t) X^{\prime}_2(t).$$
 \item $T$ is on the boundary of $S$ that is the union of the edges $(F_1,\ldots, F_n)$. It is with probability 1 not located on a vertex. Suppose that, without loss of generality, it belongs to $F_1$. Using the reasoning we have done in the preceding case, because of the emptyable property, it is easy to see that
$$ \{X(T)=u, \; X'_{\alpha}(T) \geq 0, \; X'_{\beta}(T) \leq 0\}, $$
where $\alpha$ is the upward direction on $F_1$ and $\beta$ is the inward horizontal direction.
Then, apply the Markov inequality and Rice formula in the edge $F_1$, 
 \begin{displaymath}
\begin{array}{rl}
& \PP \{\exists \, t \in F_1: \; X(t)=u, \; t\; \mbox{has minimal ordinate on} \; \mathcal{C}_u\} \\
\leq & \PP\{\exists \, t \in F_1:\; X(t)=u, \; X'_{\alpha}(t) \geq 0, \; X'_{\beta}(t) \leq 0\}\\
\leq & \E\left(\textnormal{card}\{t\in F_1:\; X(t)=u, \; X'_{\alpha}(t) \geq 0, \; X'_{\beta}(t) \leq 0\}\right) \\
= & \displaystyle \int_{F_1} \E\left(|X'_{\alpha}(t))| \mathbb{I}_{X'_{\alpha}(t) \geq 0}\mathbb{I}_{X'_{\beta}(t) \leq 0}\mid X(t)=(u)\right) \times p_{X(t)}(u) \; dt \\
= & \sigma_1(F_1) \varphi(u) \, \E\left(X'^+_{\alpha}(t)\, \mathbb{I}_{X'_{\beta}(t) \leq 0}\right).
\end{array}
\end{displaymath}
Denote by  $\theta_1$ the angle $(\alpha,\beta)$. $X'_{\beta}$ can be expressed as 
$$ \cos\theta_1\, X'_{\alpha}+\sin\theta_1\,Y, $$
with $Y$ is a standard normal variable that is independent with $X'_{\alpha}$. Then
\begin{displaymath}
\begin{array}{rl}
 &\E(X'^+_{\alpha}(t)\, \mathbb{I}_{X'_{\beta}(t) \leq 0})\\
=&\E(X'^+_{\alpha}\, \mathbb{I}_{\cos \theta_1 X'_{\alpha}+\sin \theta_1 Y \leq 0})\\
=&\displaystyle \frac{1-\cos\theta_1}{2\sqrt{2\pi}}.
\end{array}
\end{displaymath}
Summing up, the term corresponding to the boundary of $S$ is at most equal to
\begin{equation}\label{pt3}
\varphi(u) \sum_{i=1}^n \frac{(1-\cos\theta_i)\sigma_1(F_i)}{2\sqrt{2\pi}}=\frac{\varphi(u)\sigma_1(\partial S)}{2\sqrt{2\pi}},
\end{equation}
since $\displaystyle \sum_{i=1}^n \sigma_1(F_i) \, \cos\theta_i $ is just the length of the oriented projection of the boundary of $S$ on the $x$ -axis, so it is zero.
\end{itemize}
Hence, summing up (\ref{pt2}), (\ref{pt3}) and substituting into (\ref{pt1}), we obtain the  desired upper-bound in our particular case.\bigskip 

Step 2: Suppose now that $S$ is a general connected polygon such that the vertex $O$ with minimal ordinate is unique.
We define $S_1$ as the maximal emptyable subset of $S$ that contains $O$. It is easy to prove that $S_1$ is still a polygon with some horizontal edges and that $S\backslash S_1$ consists of several polygons with horizontal edges, say $S_2^1,\ldots,S_2^{n_2}$, see Figure \ref{fig2}.\\
\begin{figure}[!ht]
\begin{center}
\includegraphics[scale=0.5]{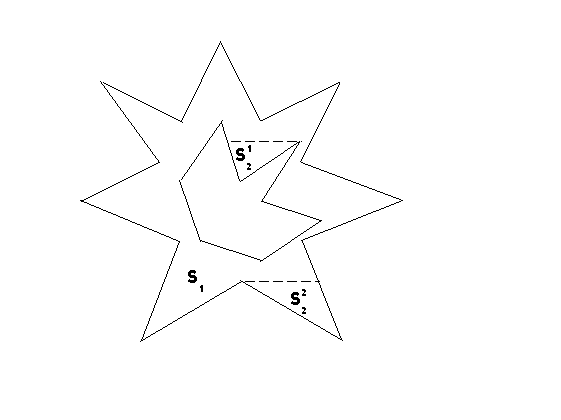}
\caption{ Example on construction of $S_1$. }\label{fig2}
\end{center}
\end{figure}
\\
So we write
\begin{equation}\label{pt4}
\PP\{M_S \geq u\}\leq \PP\{X(O)\geq u\}+\PP\{M_{S_1}\geq u,\; X(O)<u \}+\underset{i=1}{\overset{n_2}{\sum}}\PP\{M_{S_1}< u, \; M_{S_2^i}\geq u\}.
\end{equation}
Suppose for the moment that all the $S_2^i,\; i=1,\ldots,n$ are emptyable. Then, to give bounds to the event
$$\{M_{S_1}< u, \; M_{S_2^i}\geq u\},$$
we can apply the reasoning of the preceding proof but inverting the direction: in $S_2^i$, we search points on the level curve with {\bf maximum ordinate}. Let $E$ be the common edge of $S_1$ and $S_2^i$. Clearly, when $\{M_{S_1}< u, \; M_{S_2^i}\geq u\},$ the level curve is non empty and by the same arguments as in Theorem \ref{th1}, there exists $t\in S_2^i$ satisfying whether (except events with zero probability)
\begin{itemize}
\item $t$ is in the interior of $S^i_2$ and

$$\{ X(t)=u, \; X'_1(t)=0, \; X'_2(t) \leq 0, \; X''_{11}(t) \leq 0 \}.$$

From Markov inequality and Rice formula, this probability is at most equal to
\begin{equation}\label{pt5}
 \frac{\varphi (u)\sigma_2(S_2^i)}{2\pi }\left[ c\varphi (u\textrm{/}c)+u\Phi (u\textrm{/c})\right].
 \end{equation} 
\item $t$ lies on some edges of $S^i_2$. Note that $t$ can not belong to $E$. Then, as in Theorem \ref{th1}, we consider the event $t$ is on each edge and sum up the bounds to obtain
\begin{equation}\label{pt6}
\begin{array}{rl}
& \PP \left( \{\exists \, t \in \partial S_2^i: \, X(t)=u, \, t \, \mbox{has  maximal second ordinate on the level curve} \, \} \cap \{M_{S_1}<u\}\right)\\
\leq &\displaystyle \frac{\varphi(u)[\sigma_1(\partial S^i_2)-2\sigma_1(E)]}{2\sqrt{2\pi}}.
\end{array}
\end{equation}
\end{itemize} 
From (\ref{pt5}) and (\ref{pt6}) we have 
\begin{equation}\label{pt7}
\displaystyle \PP\{M_{S_1}< u, \; M_{S_2^i}\geq u\} \leq \frac{\varphi (u)\sigma_2(S_2^i)}{2\pi }\left[ c\varphi (u\textrm{/}c)+u\Phi (u\textrm{/c})\right]+\frac{\varphi(u)[\sigma_1(\partial S^i_2)-2\sigma_1(E)]}{2\sqrt{2\pi}}.
\end{equation}
Summing up all the bounds as in (\ref{pt7}), considering the upper bound for $\PP\{X(O)<u,\, M_{S_1}\geq u\}$ as in Theorem \ref{th1} and substituting into (\ref{pt4}), we get the result.\\
In the general case, when some $S_2^i$ is not emptyable, we can decompose $S_2^i$ as we did for $S$ and by induction. Since the number of vertices is decreasing, we get the result.
 
Step 3: Passing to the limit. The extension to process with non $\mathcal{C}^\infty$ paths is direct by an approximation argument. Let $\overline{X}_\epsilon(t)$ be the Gaussian field obtained by convolution of $X(t)$ with a size $\epsilon $ convolution kernel (for example a Gaussian density with variance $ \epsilon^2 I_2$). We can apply the preceding bound to the process
$$
X_ \epsilon (t) := \frac 1{\sqrt{\Var (\overline{X}_\epsilon(t)) }} \overline{X}_\epsilon\left( \Sigma ^{-1/2} _ \epsilon t\right) ,
$$
where $ \Sigma _ \epsilon = \Var(\overline{X}'_\epsilon(t))$.
Since $ \Var (\overline{X}_\epsilon(t)) \to 1$ and $ \Sigma _ \epsilon \to I_2$ , $\max_{t \in S } X_\epsilon (t) \to M_S$ and we are done. 

 The passage to the limit for $S_n$ tending to $ S$ is direct.
\end{proof}
\subsection*{Some examples}
\begin{itemize}
\item If $S$ is compact convex with non-empty interior then it is easy to construct a sequence of polygons $S_n$   converging to $S$  and such that   $\liminf_n \sigma_1(\partial S_n)=\sigma_1(\partial S)$, giving 
\begin{equation}\label{f:bound}
\PP\{M_S\geq u\}\leq \PP_R(u)=\overline{\Phi}(u)+ \frac{\sigma_1(\partial S)}{2\sqrt{2\pi}}\varphi(u)+ \frac{\sigma_2(S)}{2\pi }\left[ c\varphi (u\textrm{/}c)+u\Phi (u\textrm{/c})\right]\varphi (u).
\end{equation}
\item More generaly, if $S$  is compact and has a boundary that is piecewise-$\mathcal{C}^2$ except for a finite number of points and the closure of the interior of $S$ equals to $S$, we get  (\ref{f:bound}) by the same tools.

\item Let us now get rid of the condition $\overline{\overset{\circ}{S}}=S$  but still assuming the piecewise-$\mathcal{C}^2$ condition. Define the ``outer Minkowski content" of a closed subset $S \subset \R^2$ as (see \cite{8})
$$\rm{OMC}(S)= \underset{\epsilon \rightarrow 0}{\lim} \frac{\sigma_2(S^{+\epsilon}\setminus S)}{\epsilon},$$
whenever the limit exists (for more treatment in this subject, see \cite{2}). This definition of the perimeter differs from the quantity $\sigma_1(\partial S)$. A simple counter-example is a set corresponding to the preceding example with some ``whisker'' added. 
Using approximation by polygons, we get
\begin{equation} \label{star}
\PP\{M_S\geq u\}\leq \PP_R(u)=\overline{\Phi}(u)+ \frac{\rm{OMC}( S)}{2\sqrt{2\pi}}\varphi(u)+ \frac{\sigma_2(S)}{2\pi }\left[ c\varphi (u\textrm{/}c)+u\Phi (u\textrm{/c})\right]\varphi (u).
\end{equation}
\item The next generalization concerns  compact  $r$-convex sets  with a positive $r$ in the sense of \cite{8}.   These sets satisfy 
$$S=\underset{\overset{\circ}{B}(x,r) \cap S=\emptyset}{\bigcap}\R^2\setminus\overset{\circ}{B}(x,r).
$$
This  condition is slightly more general than the condition of having  positive reach in the sense of Federer \cite{9}.
Suppose in addition  that  $S$ satisfies the interior local connectivity property: there exists $\alpha_0 >0$ such that for all $0<\alpha<\alpha_0$ and for all $x\in S,\; \mbox{int}\left(B(x,\alpha)\cap S\right)$ is a non-empty connected set. Then we can construct a sequence of approximating polygons in the following way. 

Let $X_1,X_2,\ldots ,X_n$ be a random sample drawn from a uniform distribution on $S$ and $S_n$ be the r-convex hull of this sample, i.e
$$S_n=\underset{\overset{\circ}{B}(x,r) \cap \{X_1,X_2,\ldots , X_n\}=\emptyset}{\bigcap}\R^2\setminus\overset{\circ}{B}(x,r),$$
which can be approximated  by polygons with an arbitrary error.
By Theorem 6 of Cuevas et al \cite{8}, $S_n$ is a fully consistent estimator of $S$, it means that $d_H(S_n,S)$ and $d_H(\partial S_n,\partial S)$ tend to 0 as $n$ tends to infinity. This implies $\sigma_2(S_n)\rightarrow \sigma_2(S)$ and $\rm{OMC}(S_n)\rightarrow \rm{OMC}(S)$. Hence, we obtain (\ref{star}).
\item A complicated case: a  ``Swiss cheese''. Here, we consider an unit square and inside it, we remove a sequence of disjoint disks of radius $r_i$ such that $\displaystyle \pi \underset{i=1}
{\overset{\infty}{\sum}}r_i^2 <1$ to obtain the set $S$.
When $\displaystyle \underset{i=1}{\overset{\infty}{\sum}}r_i <\infty$ the bound (\ref{bou1}) makes sense directly.
But examples can be constructed from the Sierpinski carpet  (see Figure \ref{fig3}) such that $\displaystyle \underset{i=1}{\overset{\infty}{\sum}}r_i =\infty$ :
 divide the square into 9 subsquares of the same size and instead of removing  the central square, remove the disk inscribed in this square and  do the same procedure for the remaining 8 subsquares, ad infinitum.\\
\begin{figure}[!ht]
\begin{center}
\includegraphics[scale=0.3]{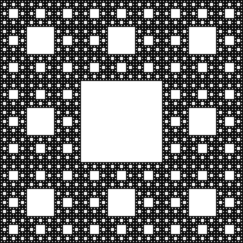}
\caption{ Sierpinski carpet (source: Wikipedia).} \label{fig3}
\end{center}
\end{figure}\\
 In our  case, 
$$\sum_{i=1}^{\infty}r_i^2=\frac{1}{4}\sum_{i=1}^{\infty}\frac{8^{i-1}}{3^{2i}} =\frac{1}{4}$$
This proves that the obtained set  $S$ has positive Lebesgue measure and is not fractal. We have on the other hand 
 $$
 \sum_{i=1}^{\infty}r_i=\frac{1}{2}\sum_{i=1}^{\infty}\frac{8^{i-1}}{3^{i}} = \infty.
 $$
Let $S_n$ be the set obtained after removing the n-th disk. Since $S \subset S_n$,
$$\PP \{M_S \geq u\} \leq \PP \{M_{S_n } \geq u\} \leq \overline{\Phi}(u)+ \frac{ \varphi (u)}{2\sqrt{2\pi}}(4+2\pi\underset{i=1}{\overset{n}{\sum}}r_i)+(1-\pi\underset{i=1}{\overset{n}{\sum}}r_i^2)\left[ c\varphi (u\textrm{/}c)+u\Phi (u\textrm{/c})\right]\varphi(u)/(2\pi).$$
Hence,
$$\PP \{M_S \geq u\} \leq \overline{\Phi}(u)+ \underset{n}{\min}\left[\frac{ \varphi (u)}{2\sqrt{2\pi}}(4+2\pi\underset{i=1}{\overset{n}{\sum}}r_i)+(1-\pi\underset{i=1}{\overset{n}{\sum}}r_i^2)\left[ c\varphi (u\textrm{/}c)+u\Phi (u\textrm{/c})\right]\varphi(u)/(2\pi)\right].$$
\end{itemize}

{\bf Remarks}: 
\begin{itemize}
 \item[1.] In comparison with other results, all the examples considered here are new. Firstly the conditions on the process are minimal and weaker than the ones of the other methods. Secondly the considered sets are not covered by any other methods. Even for the first example, because we do not assume that the number of irregular points is finite, which is needed, for example, for the convex set to be a stratified manifold as in \cite{1}.
\item[2.] Theorem \ref{th1} can be extended directly to non connected sets using sub-additivity 
 $$
 \PP\{M_{S_1\cup S_2}\geq u \} \leq \PP\{M_{S_1}\geq u \} + \PP\{M_{ S_2}\geq u \}.
 $$ 
 This implies that the coefficient of $\overline{\Phi}(u)$ in (\ref{bou1}) must be the number of components.

\end{itemize}
 
\subsection*{Is the bound sharp?}
\begin{itemize}
\item Under Assumption 2, Adler and Taylor \cite{1} show that
$$\underset{u\rightarrow +\infty}{\liminf} -2u^2\log |\PP\{M_S\geq u\} -\PP_E(u)|\geq 1+1/c^2.$$
From
$$0\leq \PP_R(u)-\PP_E(u)=\frac{\sigma_2(S)}{2\pi}\varphi(u)\left[c\varphi (u\textrm{/}c)-u\overline{\Phi} (u\textrm{/}c)\right]$$
and the elementary inequality for $x>0$,
$$\varphi(x)\left(\frac{1}{x}-\frac{1}{x^3}\right)<\overline{\Phi}(x)<\varphi(x)\left(\frac{1}{x}-\frac{1}{x^3}+\frac{3}{x^5}\right),$$ it is easy to see that
$$\underset{u\rightarrow +\infty}{\lim\inf} -2u^2\log (\PP_R(u)-\PP_E(u)) \geq 1+1/c^2.$$
 So the upper bound $\PP_R(u)$ is as sharp as $\PP_E(u)$ .
 \item Let $S$ be a compact and simply connected domain in $\R^2$ having a $\mathcal{C}^3$-piecewise boundary. Assume that all the discontinuity point are convex, in the sense that if we parametrize the boundary in the direction of positive rotation, then at each discontinuity point, the angle of the tangent has a positive discontinuity. Then, it is easy to see that the quantity
$$\kappa (S)= \underset{t\in S}{\sup} \underset{s\in S,\; s\neq t}{\sup} \frac{\rm{dist}(s-t,\mathcal{C}_t)}{\|s-t\|^2}$$
is finite, where $\rm{dist}$ is the Euclidean distance and $C_t$ is the cone generated by the set of directions
 $$\Biggl\{ \lambda \in \R^2 :\; \|\lambda\|=1,\, \exists s_n \in S \; \rm{such\; that}\; s_n\rightarrow t \; \rm{and} \; \frac{s_n-t}{\|s_n-t\|}\rightarrow \lambda \Biggr\}.$$
In order to apply the Theorem 8.12 in \cite{5}, besides the Assumption 1, we make some additional assumptions on the field $X$ such that it satisfies the conditions (A1)-(A5) page 185 in \cite{5}. Assume that 
\begin{itemize}
\item $X$ has $\mathcal{C}^3$ paths.
\item The covariance function $r(t)$ satisfies $|r(t)|\neq 1$ for all $t\neq 0$.
\item For all $s\neq t$, the distribution of $(X(s),X(t),X'(s),X'(t))$ does not degenerate.
\end{itemize} 
With these hypotheses, we can see that
\begin{itemize}
\item The conditions (A1)-(A3) are easily verified.
\item The condition (A4) which states that the maximum is attained at a single point, can be deduced from Proposition 6.11 in \cite{5} since for $s\neq t$, $(X(s),X(t),X'(s),X'(t))$ has a nondegenerate distribution. 
\item The condition (A5) which states that almost surely there is no point $t\in S$ such that $X'(t)=0$ and $\det(X''(t))=0$, can be deduced from Proposition 6.5 in \cite{5} applied to the process $X'(t)$.
\end{itemize}
Since all the required conditions are met, by Theorem 8.12 in \cite{5}, we have
\begin{equation}\label{f:1}
\underset{x\rightarrow +\infty}{\liminf}-2x^2 \log \big[\PP_M(x)-\PP\{M_S\geq x\}\big] \geq 1+ \underset{t\in S}{\inf}\frac{1}{\sigma_t^2+\kappa_t^2}>1,
\end{equation}
where
$$\sigma_t^2=\underset{s\in S\setminus\{t\}}{\sup}\frac{\rm{Var}\left(X(s)\mid X(t),X'(t)\right)}{(1-r(s,t))^2}$$
and
$$\kappa_t=\underset{s\in S\setminus\{t\}}{\sup}\frac{\rm{dist}\left(\frac{\partial}{\partial t} \it{r}(s,t),C_t\right)}{1-r(s,t)}.$$
Note that the condition $\kappa(S)$ is finite implies that $\kappa(t)$ is also finite for every $t\in S$. (\ref{f:1})  is true also for $\PP_R$, since as $x\rightarrow +\infty$, $\PP_R(x)$ is smaller than $\PP_M(x)$ (see Section \ref{se5} for the easy proof). As a consequence $\PP_R$ is super exponentially sharp.

\item Suppose that $S$ is a circle in $\R^2$. Then $\{X(t)\, :\; t\in S\}$ can be viewed as a periodic process on the line. In that case, it is easy to show, see for example Exercise 4.2 in \cite{5}, that as $u\rightarrow \infty$
$$
\PP(M_S\geq u)=\frac{\sigma_1(S)}{\sqrt{2\pi}}\varphi(u) + O(\varphi(u(1+\delta))=\frac{\rm{OMC}(S)}{2\sqrt{2\pi}}\varphi(u) + O(\varphi(u(1+\delta))
$$
for some $\delta >0$; while Theorem \ref{th1} gives with a standard approximation of the circle by polygons 
$$\PP(M_S\geq u) \leq P_R(u)=\overline{\Phi}(u)+ \frac{\rm{OMC}(S)}{2\sqrt{2\pi}}\varphi(u),
$$
which is too large. 
This shows that the bound $P_R$ is not always super exponentially  sharp.
\end{itemize}

\section{The record method in dimension 3}\label{se3}

  For example, with  the direct method,  some difficulties arise in dimension 3 because we need to compute 
$$\E | \det (X'' (t)|,$$
 under some conditional law. This can be conducted only in the isotropic case using random matrices theory, see \cite{4} and even in this case the result is complicated. In dimension 2, the record method is a trick that permits to spare a dimension in the size of the determinant we have to consider because the conditioning implies a factorization. For example in equation (\ref{pt2}) we have used the fact that 
 $$ \det( Z'(t) ) = X''_{11}(t) X^{\prime }_2(t),$$
 under the condition. In this section we will use the same kind of trick to pass from a 3,3 matrix to a 2,2 matrix and then a 2,2 determinant is just a quadratic form so we can use, to compute the expectation of its absolute value, the Fourier method of Berry and Dennis \cite{6} or Li and Wei \cite{11}. This computation is detailed in Section 4 and is one of the main contributions of the paper. \\
 Before stating the main theorem of this section, we recall the following lemma (see Chapter 5 of Prasolov and Sharygin \cite{15}) \bigskip
 \begin{lemma}\label{lema}
  Let $Oxyz$ be a trihedral. Denote by $a,\, b$ and $c$ the plane angles $\widehat{xOy},\, \widehat{yOz}$ and $\widehat{zOx}$, respectively. Denote by $A, \,B$ and $ C$ the angles between two faces containing the line $Oz, \,  Ox$ and $ Oy$, respectively. Then, 
\begin{itemize}
\item[a.] $\sin a : \sin A=\sin b : \sin B=\sin c : sin C$.
\item[b.] $\cos a=\cos b\cos c+\sin b\sin c \cos A$.
\end{itemize}
\end{lemma} 
Our main result is the following 
 
\begin{theorem}\label{th3}
Let $S$ be a compact and convex subset of $\R^3$  with non-empty  interior  and  let $X$ satisfy Assumption 1. Suppose,  in addition  that
 $X$ is  isotropic with respect to the first and second coordinate, i.e 
 $$\textnormal{Cov}(X(t_1,t_2,t_3);X(s_1,s_2,t_3))=\rho((t_1-s_1)^2+(t_2-s_2)^2) \mbox{ with}\; \rho \, \mbox{ of class}\;  \mathcal{C}^2.$$

 Then, for every real $u$, 
\begin{displaymath}
\begin{array}{rl}
\PP \{ M\geq u \} \leq & 1 -\Phi (u) + \displaystyle \frac{2\lambda(S)}{\sqrt{2\pi }} \varphi (u) + \frac{\sigma_2(S) \varphi (u)}{4\pi } \left[ \sqrt{12\rho''(0) -1}\varphi \left(\frac{u}{\sqrt{12\rho''(0) -1}} \right) + u \Phi \left(\frac{u}{\sqrt{12\rho''(0) -1}} \right)\right]\\
&+ \displaystyle \frac{\sigma_3(S)\varphi (u)}{(2\pi)^{3/2} } \left[u^2-1+\frac{\displaystyle (8\rho''(0))^{3/2}\exp\left(-u^2.(24\rho''(0)-2)^{-1}\right)}{\displaystyle \sqrt{24\rho''(0)-2}}\right],
\end{array}
\end{displaymath}
where $\lambda$ is the caliper diameter.
\end{theorem}
\begin{proof}
By the same limit argument as in Theorem \ref{th1}, we can assume that $X(t) $ has $\mathcal{C}^{\infty}$ paths and that  $S$ is a convex polyhedron. Let $O$ be the vertex of $S$ that has minimal third coordinate, we can assume also that this vertex is unique. It is clear that if $X(O) < u$ and $M_S>u$ then the level set 
$$\mathcal{C}(u)=\{t\in S: \; X(t)=u\}$$
is non empty and there exists  at least one  point $T$ having minimal third coordinate on this set. Then,
\begin{equation}
\begin{array}{rcl}
\PP\{M_S\geq u\}& = & \PP \{ X(O)\geq u \}+\PP\{X(O)<u,\, M_S\geq u\}\\
&\leq & \PP \{ X(O)\geq u \}+\PP\{\exists \,  T\in S: \, X(T)=u,\, T \, \mbox{has minimal third coordinate on}\; \mathcal{C}_u\}.
\end{array}
\end{equation}
Now, we consider three possibilities:\\
$\bullet $ Firstly, if $T$ is in the interior of $S$, then by the same arguments as in Theorem \ref{th1}, for all the point $s\in S$ with the third coordinate smaller than the one of $T$, $X(s)<X(T)$; 
it means that, at $T$, $X(t)$ has a local 
maximum
  with respect to the first and second coordinates and  is non-decreasing with respect to the third coordinate. Therefore, setting 
 \begin{displaymath}
A(t)= \left( 
\begin{array}{cc}
X''_{11}(t) & X''_{12}(t) \\
X''_{12}(t) & X''_{22}(t)
\end{array}
\right),
\end{displaymath}
we have 
 $$ 
 \{X(T)=u, \;X'_1(T)=0, \; X'_2(T)=0,\; A(T) \preceq 0,\; X'_3(T) \geq 0\}.
 $$
Then, apply the Rice formula to the field $Z=(X,X'_1,X'_2)$ and the Markov inequality, 
\begin{displaymath}
\begin{array}{rl}
&\displaystyle\PP\{\exists \,  T\in \overset{\circ}{S}: \, X(T)=u,\, T \, \mbox{has  minimal third coordinate on}\; \mathcal{C}_u\}\\
 \leq &\displaystyle \PP\{\exists \,  t \in \overset{\circ}{S}:\;X(t)=u, \; X'_1(t)=0, \; X'_2(t)= 0, \; X'_3(t)\geq 0,\; A(t) \preceq 0\}\\
\leq &\displaystyle \E\big(\textnormal{card}\{t\in \overset{\circ}{S}:\; X(t)=u, \; X'_1(t)=0, \; X'_2(t)= 0, \; X'_3(t)\geq 0,\; A(t) \preceq 0\}\big) \\
= &\displaystyle \E\big(\textnormal{card}\{t\in \overset{\circ}{S}:\;Z(t)=(u,0,0), \; X'_3(t) \geq 0, \; A(t) \preceq 0\}\big)\\
=& \displaystyle \int_{\overset{\circ}{S}} \E\left(|\det(Z'(t))|\mathbb{I}_{X'_3(t) \leq 0}\mathbb{I}_{A(t)\preceq 0} \mid Z(t)=(u,0,0)\right) \times p_{Z(t)}(u,0,0)\; dt.\\

\end{array}
\end{displaymath}
Under the condition $Z(t)=(u,0,0)$, it is clear that $\det(Z'(t))=X'_3(t)\, \det(A(t))$. So, we obtain the bound
$$ \sigma_3(S) \, \frac{\varphi(u)}{2\pi} \, \E\left(|\det(A(t))|.\mathbb{I}_{A(t) \preceq 0}X'^+_3(t) \mid Z(t)=(u,0,0)\right).
$$
From Corollary \ref{cor} of Section \ref{se4}, we know that 
$$
\E\left(|\det(A(t))|. \mathbb{I}_{A(t)\preceq 0}\mid Z(t)=(u,0,0) \right) \leq u^2-1+\frac{\displaystyle (8\rho''(0) )^{3/2}\exp\left(-u^2.(24\rho''(0) -2)^{-1}\right)}{\displaystyle \sqrt{24\rho''(0) -2}}.$$
Hence,
\begin{align}
&\PP\{\exists \,  T\in \overset{\circ}{S}: \, X(T)=u,\, T \, \mbox{has  minimal third coordinate on}\; \mathcal{C}_u\}  \notag
\\
\leq \;\;&
\label{t1}
 \frac{\sigma_3(S)\varphi (u)}{(2\pi)^{3/2} } \left[u^2-1+\frac{\displaystyle (8\rho''(0))^{3/2}\exp\left(-u^2.(24\rho''(0)-2)^{-1}\right)}{\displaystyle \sqrt{24\rho''(0)-2}}\right]. 
 \end{align}
$\bullet $ Secondly, if $T$ is in the interior of a face $S_1$, then, in this face, we choose the base $\overrightarrow{\alpha},\, \overrightarrow{\beta}$ such that $\overrightarrow{\alpha}$ is in the horizontal plane  $ 0\ t_1t_2$ such that along this vector, the second coordinate is not decreasing. Let us denote vector $\overrightarrow{\gamma}$ in the horizontal plane that is perpendicular to $\alpha$ and goes into $S$. It is easy to see that
$$\{ X(T)=u, \; X'_ {\alpha}(T) =0,\; X'_{\beta}(T) \geq 0,\; X'_{\gamma}(T) \leq 0,\; X''_{\alpha}(T) \leq 0 \}.$$
Apply Markov inequality and Rice formula to the field $Y(t)=(X(t),X'_{\alpha}(t))$, 
\begin{displaymath}
\begin{array}{rl}
&\displaystyle \PP\{\exists \,  T\in \overset{\circ}{S_1}: \, X(T)=u,\, T \, \mbox{has the minimal third ordinate on}\; \mathcal{C}_u\}\\
 \leq & \PP \{\exists \,  t \in \overset{\circ}{S_1}: \; X(t)=u, \; X'_ {\alpha}(t) =0,\; X'_{\beta}(t) \geq 0,\; X'_{\gamma}(t) \leq 0,\; X''_{\alpha} \leq 0 \}\\
\leq & \E (\textnormal{card}\{ t \in \overset{\circ}{S_1}: \; X(t)=u, \; X'_ {\alpha}(t) =0,\; X'_{\beta}(t) \geq 0,\; X'_{\gamma}(t) \leq 0,\; X''_{\alpha} \leq 0 \})\\
= & \displaystyle \int_{\overset{\circ}{S_1}} \E\left( |\det(Y'(t))| \mathbb{I}_{X'_{\beta}(t) \geq 0}\mathbb{I}_{X'_{\gamma}(t) \leq 0}\mathbb{I}_{X''_{\alpha}(t) \leq 0}\mid Y(t)=(u,0)\right)\, p_{Y(t)}(u,0)\, dt\\
= &\displaystyle \frac{\sigma_2(S_1)\varphi(u)}{\sqrt{2\pi}} \E\left( |X''^-_{\alpha}(t)|X'^+_{\beta}(t)\mathbb{I}_{X'_{\gamma}(t) \leq 0} \mid Y(t)=(u,0)\right).
\end{array} 
\end{displaymath}
As in Theorem \ref{th1}, it is clear that
\begin{align*}
\E\big( |X''^-_{\alpha}(t)| \mid Y(t)=(u,0)\big) & =\, \sqrt{12\rho''(0) -1}\varphi \left(\frac{u}{\sqrt{12\rho''(0) -1}} \right) + u \Phi \left(\frac{u}{\sqrt{12\rho''(0) -1}} \right),\\
 \E\big( X'^+_{\beta}(t)\mathbb{I}_{X'_{\gamma}(t) \leq 0} \mid Y(t) &=(u,0)\big)=\frac{1-\cos(\beta,\gamma)}{2\sqrt{2\pi}}.
\end{align*}
Observe that the angle between $\beta$ and $\gamma$ is the angle $\theta_1$ between the face $S_1$ and the horizontal plane, then the probability that there exists one point with minimal third coordinate on the level set and in the interior of the face $S_1$ is at most equal to
$$ \frac{\sigma_2(S_1) \varphi (u)(1-\cos\theta_1)}{4\pi } \left[ \sqrt{12\rho''(0) -1}\varphi \left(\frac{u}{\sqrt{12\rho''(0) -1}} \right) + u \Phi \left(\frac{u}{\sqrt{12\rho''(0) -1}} \right)\right] .$$
Taking the sum of all the bounds at each faces, observing that 
$$\sum_{i=1}^n \sigma_2(S_i)\cos\theta_i=0,$$ 
we have   the following upper bound for the probability  of having  a point  $T$  with minimal third coordinate  on the level set  and belonging  to  a face:
\begin{equation}\label{t2}
 \frac{\sigma_2(S) \varphi (u)}{4\pi } \left[ \sqrt{12\rho''(0) -1}\varphi \left(\frac{u}{\sqrt{12\rho''(0) -1}} \right) + u \Phi \left(\frac{u}{\sqrt{12\rho''(0) -1}} \right)\right] .
\end{equation}

$\bullet $ Thirdly, when $T$ belongs to one edge, for example $F_1$.
Let us define $\overrightarrow{\eta}$ is the upward direction on this edge, i.e such that along this vector, the third coordinate is not decreasing, and $\overrightarrow{\alpha}$ and $\overrightarrow{\beta}$ are the two horizontal directions that go inside two faces containing the edge. Then, 
$$ \{X(T)=u,\; X'_{\eta}(T) \geq 0,\; X'_{\alpha}(T) \leq 0,\; X'_{\beta}(T) \leq 0\} .$$
By Rice formula, the expectation of the number of the points in $F_1$ satisfying this condition is
\begin{displaymath}
\begin{array}{rl}
& \displaystyle \int_{F_1} \E\left(X'^+_{\eta}(t) \mathbb{I}_{X'_{\alpha}(t) \leq 0}\mathbb{I}_{X'_{\beta}(t) \leq 0} \mid X(t)=u\right)\times p_{X(t)}(u) \; dt\\
=& \displaystyle \sigma_1(F_1) \; \varphi(u) \; \E\left(X'^+_{\eta}(t) \mathbb{I}_{X'_{\alpha}(t) \leq 0}\mathbb{I}_{X'_{\beta}(t) \leq 0}\right). 
\end{array}
\end{displaymath}
Let $\overrightarrow{a}$ and $\overrightarrow{b}$ be two vectors in two faces containing the edge $F_1$ and perpendicular to $\overrightarrow{\eta}$; $\theta_1$ be the angle between $\overrightarrow{\alpha}$ and $\overrightarrow{\eta}$; $\theta_2$ be the angle between $\overrightarrow{\beta}$ and $\overrightarrow{\eta}$. It is clear that
$$X'_{\alpha}(t)=\cos\theta_1 \, X'_{\eta}(t)+ \sin\theta_1 \, X'_a(t) ,$$
$$ X'_{\beta}(t)=\cos\theta_2 \, X'_{\eta}(t)+\sin\theta_2 \, X'_b(t), $$
and $\textnormal{cov}(X'_a(t),X'_b(t))=\cos\theta_3$, where $\theta_3$ is the angle between two faces containing the edge $F_1$. Then,
\begin{displaymath}
\begin{aligned}
&\E\left(X'^+_{\eta} \mathbb{I}_{X'_{\alpha}(t) \leq 0}\mathbb{I}_{X'_{\beta}(t) \leq 0}\right)\\
=& \E\left(X'^+_{\eta} \mathbb{I}_{\{\cos\theta_1 \, X'_{\eta}(t)+ \sin\theta_1 \, X'_a(t) \leq 0\} }\mathbb{I}_{\{\cos\theta_2 \, X'_{\eta}(t)+\sin\theta_2 \, X'_b(t) \leq 0\}}\right)\\
=& \displaystyle \int_0^{\infty} x \, \varphi(x) \, F(x) \, dx,\\
\end{aligned}
\end{displaymath}
where
\begin{displaymath}
\begin{aligned}
F(x)= & \E\left( \mathbb{I}_{\{\cos\theta_1 \, X'_{\eta}(t)+ \sin\theta_1 \, X'_a(t) \leq 0\}}\mathbb{I}_{\{\cos\theta_2 \, X'_{\eta}(t)+\sin\theta_2 \, X'_b(t) \leq 0\}} \mid X'_{\eta}(t)=x\right)\\
= & \displaystyle \int_{-\infty}^{-\cot\theta_1 .x} \varphi(y) \, \Phi\left(\frac{- \cot\theta_2 .x - \cos\theta_3 .y}{\sin \theta_3}\right) \, dy.\\
\end{aligned}
\end{displaymath}
So, 
\begin{displaymath}
\begin{aligned}
F'(x)=&-\cot \theta_2 \; \varphi(-\cot\theta_2 \, .x) \, \Phi\left(\frac{-\cot\theta_1 \,.x+\cos \theta_3 \, \cot \theta_2\, . x}{\sin \theta_3}\right)\\
&-\cot \theta_1 \; \varphi(-\cot (\theta_1 \, x)) \, \Phi\left(\frac{-\cot\theta_2\, . x+\cos \theta_3 \, \cot \theta_1\, . x}{\sin \theta_3}\right).
\end{aligned}
\end{displaymath}
By integration by parts, 
\begin{displaymath}
\begin{aligned}
\displaystyle \int_0^{\infty} x\, \varphi(x)\, F(x)\, dx & = -\int_0^{\infty} F(x) \, d(\varphi(x))\\
&= \displaystyle F(0)\varphi(0)+ \int_0^{\infty} \varphi(x)\, F'(x)\, dx.
\end{aligned}
\end{displaymath}
It is easy to check that
 $$\int_0^{\infty} \varphi(x)\, \Phi(mx) \, dx=\frac{1}{4}+\frac{\arctan (m)}{2\pi},$$
$$\int_{-\infty}^0 \varphi(x)\, \Phi(mx) \, dx=\frac{1}{4}-\frac{\arctan (m)}{2\pi}.$$
From the above results, we have
\begin{displaymath}
\begin{aligned}
F(0)\varphi(0)& = \frac{1}{\sqrt{2\pi}} \int_{-\infty}^0 \varphi(y)\Phi\left( \frac{-\cos\theta_3 \, }{\sin\theta_3}y\right) \, dy\\
&= \frac{1}{(2\pi)^{3/2}}\left(\pi-\theta_3\right).
\end{aligned}
\end{displaymath}
and
\begin{displaymath}
\begin{aligned}
\int_0^{\infty} \varphi(x)\, F'(x)\, dx&= -\cot \theta_2 \,\int_0^{\infty} \varphi(x) \varphi(-\cot \theta_2\, . x) \, \Phi\left(\frac{-\cot \theta_1\, . x+\cos \theta_3 \, \cot \theta_2\, . x}{\sin \theta_3}\right) \, dx\\
&\quad -\cot \theta_1 \,\int_0^{\infty} \varphi(x) \varphi(-\cot \theta_1\, . x) \, \Phi\left(\frac{-\cot \theta_2\, .x+\cos \theta_3 \, \cot \theta_1\, . x}{\sin \theta_3}\right)\,dx\\
&= \frac{-\cos\theta_2}{\sqrt{2\pi}}\left(\frac{1}{4}+\frac{1}{2\pi}\arctan \left(\frac{-\sin\theta_2.\cot\theta_1+\cos\theta_3\cos\theta_2}{\sin\theta_3}\right)\right)\\
&\quad + \frac{-\cos\theta_1}{\sqrt{2\pi}}\left(\frac{1}{4}+\frac{1}{2\pi}\arctan \left(\frac{-\sin\theta_1\cot\theta_2+\cos\theta_3\cos\theta_1}{\sin\theta_3}\right)\right).
\end{aligned}
\end{displaymath}
Therefore, the probability that there exists one point with minimal third coordinate on the level set $\mathcal{C}(u)$ and belonging to $F_1$ is at most equal to
\begin{displaymath}
\begin{array}{rl}
 \displaystyle \frac{\sigma_1(F_1)(\pi-\theta_3)\varphi(u)}{(2\pi)^{3/2}} +\sigma_1(F_1)& \displaystyle \left[\frac{-\cos\theta_2}{\sqrt{2\pi}} \left( \frac{1}{4}+\frac{1}{2\pi}\arctan \left(\frac{-\sin\theta_2\cot\theta_1+\cos\theta_3\cos\theta_2}{\sin\theta_3}\right)\right) \right.\\
 &\displaystyle +\left. \frac{-\cos\theta_1}{\sqrt{2\pi}}\left(\frac{1}{4}+\frac{1}{2\pi}\arctan \left(\frac{-\sin\theta_1\cot\theta_2+\cos\theta_3\cos\theta_1}{\sin\theta_3}\right)\right)\right].
\end{array}
\end{displaymath}
Summing up all the terms at all the edges, we obtain the bound
\begin{displaymath}
\begin{aligned}
\displaystyle \varphi(u)\sum_{i=1}^n \frac{\sigma_1(F_i)(\pi-\theta_{3i})}{(2\pi)^{3/2}} & +\displaystyle \varphi(u)\sum_{i=1}^n \sigma_1(F_i)\left[ \frac{-\cos\theta_{2i}}{\sqrt{2\pi}}\left( \frac{1}{4}+\frac{1}{2\pi}\arctan \left(\frac{-\sin\theta_{2i}\cot\theta_{1i}+\cos\theta_{3i}\cos\theta_{2i}}{\sin\theta_{3i}}\right)\right) \right.\\
&\displaystyle \qquad + \left. \frac{-\cos\theta_{1i}}{\sqrt{2\pi}}\left(\frac{1}{4}+\frac{1}{2\pi}\arctan \left(\frac{-\sin\theta_{1i}\cot\theta_{2i}+\cos\theta_{3i}\cos\theta_{1i}}{\sin\theta_{3i}}\right)\right)\right].
\end{aligned}
\end{displaymath}
By definition, 
$$\sum_{i=1}^n \sigma_1(F_i) (\pi-\theta_{3i})=4\pi\lambda(S).$$
Now, we prove 
\begin{displaymath}
\begin{aligned}
I = & \displaystyle \sum_{i=1}^n l_i \left[ \frac{\cos\theta_{2i}}{\sqrt{2\pi}}\left(\frac{1}{4}+\frac{1}{2\pi}\arctan \left(\frac{-\sin\theta_{2i}\cot\theta_{1i}+\cos\theta_{3i}\cos\theta_{2i}}{\sin\theta_{3i}}\right)\right) \right.\\
&\qquad + \left. \frac{\cos\theta_{1i}}{\sqrt{2\pi}}\left(\frac{1}{4}+\frac{1}{2\pi}\arctan \left(\frac{-\sin\theta_{1i}\cot\theta_{2i}+\cos\theta_{3i}\cos\theta_{1i}}{\sin\theta_{3i}}\right)\right) \right]=0. 
\end{aligned} 
\end{displaymath}
Indeed, from Lemma \ref{lema}, we have
$$ \frac{-\sin\theta_1\cot\theta_2+\cos\theta_3\cos\theta_1}{\sin\theta_3}=\frac{-\cos h}{\sin h}, $$
where $h$ is the dihedral angle at $\overrightarrow{\alpha}$, i.e, the angle between the horizontal plane and the face containing $\overrightarrow{\alpha}$ and $\overrightarrow{\eta}$. Since $h$ is constant for each face,
$$ I= \sum_{S \in \{S_1,\ldots,S_k\}} \sum_{l \subset S} l\cos \theta_1 \left(\frac{1}{4}+\frac{1}{2\pi}\arctan\left(\frac{-\cos h}{\sin h}\right)\right)=0. $$
Therefore, we have the following upper bound for the probability  of having  a point  $T$  with minimal third coordinate  on the level set  and belonging  to  an edge:
\begin{equation}\label{t3}
\frac{2\lambda(S)\varphi(u)}{(2\pi)^{1/2}}.
\end{equation}
From (\ref{t1}), (\ref{t2}), (\ref{t3}) and the fact that $\PP\{X(O)>u\}=\overline{\Phi}(u)$, the result follows.
\end{proof}
\section{Computation of the absolute value of the determinant of the Hessian matrices}\label{se4}
As we see in the proof of Theorem \ref{th3}, we deal with the following 
$$ \E(|\det(X''(t))|\mathbb{I}_{X''(t)\preceq 0} \mid X(t)=u,\; X'_1(t)=0,\;X'_2(t)=0). $$
To evaluate this quantity, we have the following statement that is one of our main results in this paper:
\begin{theorem}
Let $X$ be a standard stationary isotropic centered two-dimensional Gaussian field. One has
\begin{equation}\label{eth3}
\E\left(|\det(X''(t))|\mid (X,X'_1,X'_2)(t)=(u,0,0) \right) \, = u^2-1+2\frac{\displaystyle (8\rho''(0) )^{3/2}\exp(-u^2.(24\rho''(0) -2)^{-1})}{\displaystyle \sqrt{24\rho''(0) -2}}.
\end{equation}
\end{theorem}
\begin{proof}
Under the condition, the vector $(X''_{11},X''_{12},X''_{22})$ has the same distribution with $(Y_1,Y_2,Y_3)+(-u,0,-u)$, where $(Y_1,Y_2,Y_3)$ is a centered Gaussian vector with the covariance matrix:\\
 $$\Sigma =
\left( \begin{array}{ccc}
 12\rho''(0) -1 & 0 & 4\rho''(0) -1\\
 0 & 4\rho''(0) & 0 \\
 4\rho''(0) -1& 0 & 12\rho''(0) -1 \\
\end{array} \right).$$
Then, the LHS in (\ref{eth3}) can be written as\\
$$\begin{array}{ll}
 & \E(|X''_{11}(t)X''_{22}(t)-X''_{12}(t)^2| \, \mid \, (X,X'_1,X'_2)(t)=(u,0,0) ) \\
=& \E(|(Y_1-u)(Y_3-u)-Y^2_2|)\\
=& \E(|Y_1Y_3-Y^2_2-u(Y_1+Y_3)+u^2|)\\
=& \E(|<Y,AY>+<b,Y>+u^2|),
\end{array}$$
where $A=\left(\begin{array}{ccc} 0 & 0 & \frac{1}{2}\\ 
0 & -1 & 0 \\
\frac{1}{2} & 0 & 0 \\
\end{array}\right)$ 
and $b=\left(\begin{array}{c} -u \\ 0 \\ -u \\ \end{array}\right)$.\\
Here, from Theorem 2.1 of \cite{11}, the expectation is equal to
$$\E(|<Y,AY>+<b,Y>+u^2|)=\frac{2}{\pi } \int_0^{\infty }t^{-2}(1-F(t)-\overline{F}(t))dt,$$
where $$F(t)=\frac{\exp(itu^2-2^{-1}t^2<b,(I-2it\Sigma A)^{-1}\Sigma b>)}{2\det(I-2it\Sigma A)^{1/2}}.$$
It is clear that 
$$F(t)=\frac{\exp(itu^2[1-it(16\rho''(0) -2)]^{-1})}{2(1+8it\rho''(0) )[1-it(16\rho''(0) -2)]^{1/2}},$$
and
$$\bar {F}(t)=\frac{\exp(-itu^2[1+it(16\rho''(0) -2)]^{-1})}{2(1-8it\rho''(0) )[1+it(16\rho''(0) -2)]^{1/2}}=F(-t).$$
So, the expectation is equal to
$$\begin{array}{rl}
\displaystyle \frac{2}{\pi} \int_0^{\infty }\frac{1}{t^2}(1-F(t)-\overline{F}(t))dt=& \displaystyle \textnormal{Re}( \frac{1}{\pi} \int_{-\infty}^{\infty }\frac{1}{t^2}(1-2.F(t))dt)\\

=& \displaystyle \textnormal{Re}(\frac{1}{\pi} \int_{-\infty}^{\infty }\frac{1}{t^2}(1-\frac{\displaystyle \exp(itu^2[1-it(16\rho''(0) -2)]^{-1})}{\displaystyle (1+8it\rho''(0) )[1-it(16\rho''(0) -2)]^{1/2}})dt).
\end{array}$$
Here, we apply the residue theorem to compute
$$\begin{array}{ll}
& \displaystyle \frac{1}{\pi} \int_{-\infty}^{\infty }\frac{1}{t^2}(1-\frac{\displaystyle \exp(itu^2[1-it(16\rho''(0) -2)]^{-1})}{\displaystyle (1+8it\rho''(0) )[1-it(16\rho''(0) -2)]^{1/2}})dt\\
=& 2i . \left(\text{sum of residues in upper half plane}\right) + i. \left(\text{sum of residues on x-axis} \right).
\end{array}$$
The residues come from two poles at $i.(8\rho''(0) )^{-1}$ and $0$ and we see that:\\
The residue at $0$ is equal to 
$$\frac{d}{dt}\left( 1-\frac{\exp(itu^2[1-it(16\rho''(0) -2)]^{-1})}{(1+8it\rho''(0) )[1-it(16\rho''(0) -2)]^{1/2}}\right) \bigg|_{t=0}=-i.u^2+i.$$
And the residue at $i.(8\rho''(0) )^{-1}$ is equal to
\begin{displaymath}
\begin{array}{l}
\displaystyle \frac{(1+8it\rho''(0) )[1-it(16\rho''(0) -2)]^{}-\exp(itu^2[1-it(16\rho''(0) -2)]^{-1})}{t^2.8i\rho''(0) .[1-it(16\rho''(0) -2)]^{1/2}} \bigg|_{t=i.(8\rho''(0) )^{-1}}\\
= \displaystyle \frac{(8\rho''(0) )^{3/2}\exp(-u^2.(24\rho''(0) -2)^{-1})}{\sqrt{24\rho''(0) -2}.i}.
\end{array}
\end{displaymath}
These two residues imply the result.
\end{proof}
We have the corollary
\begin{corollary} \label{cor}
Let $X$ be a standard stationary isotropic centered Gaussian field. One has
$$\E(|\det(X''(t))|. \mathbb{I}_{X''(t)\preceq 0}\mid (X,X'_1,X'_2)(t)=(u,0,0) ) \leq u^2-1+\frac{\displaystyle (8\rho''(0) )^{3/2}\exp(-u^2.(24\rho''(0) -2)^{-1})}{\displaystyle \sqrt{24\rho''(0) -2}}.$$
\end{corollary}
\begin{proof}
The result follows from two observations
\begin{itemize}
\item $\displaystyle |\det(X''(t))|. \mathbb{I}_{X''(t)\preceq 0} \leq \frac{|\det(X''(t))|+\det(X''(t))}{2}.$
\item $\E(\det(X''(t))\mid (X,X'_1,X'_2)(t)=(u,0,0)) \; = \; u^2-1.$
\end{itemize}
\end{proof}

\section{Numerical comparison}\label{se5}
In this section, we compare the upper bounds given by the direct method and record method with the approximation given by the EC method.
For simplicity  we limit our attention  to the case  where  $S$ is the square $[0,T]^2$ and $X$ is a standard stationary isotropic centered Gaussian field with  covariance function $\rho(\|s-t\|^2)$. Note that only $\rho''(0)$ plays a role, the exact form of $\rho$ does not need to be specified. More precisely, we consider
 \begin{itemize}
 \item[1.] the approximation given by the EC method
 $$
 \PP_E (u) = 
 \overline{\Phi}(u)+\frac{2T}{\sqrt{2\pi}} \varphi (u)+\frac{T^2}{2\pi}u\varphi (u);
 $$
\item[2.] and the upper bound given by the direct method 
\begin{align}
  \PP_M (u) = \quad \overline{\Phi}(u)\displaystyle &+\frac{2T}{\sqrt{2\pi}} \int_u^{\infty} \left[ c\varphi (x\textrm{/}c)+x\Phi (x\textrm{/c})\right]\varphi (x)dx \notag \\
&\displaystyle + \frac{T^2}{2\pi } \int_u^{\infty} \left[x^2-1+\left(\frac{2(c^2+1)}{3}\right)^{3/2} \sqrt{\pi}\frac{ \varphi(x/c)}{c}\right]\varphi(x)dx, \notag
\end{align}
where $c=\sqrt{12\rho''(0)-1}$,\\

\item[3.] and the one given by the record method
$$\PP_R(u) = \overline{\Phi}(u)+\frac{2T}{\sqrt{2\pi}} \varphi (u)+\frac{T^2}{2\pi }\left[ c\varphi (u\textrm{/}c)+u\Phi (u\textrm{/c})\right]\varphi (u).
$$
\end{itemize}
It is easy to see  that $\PP_E$ is always less than $\PP_R$ and $\PP_M$. We will prove that $\PP_R(u) $ is smaller than $\PP_M (u) $ as $u$ is large. Indeed, if we compare the "dimension 1 terms" (corresponding to $\sigma_1(\partial S)$), we have
 \begin{displaymath}
\begin{array}{rl}
&  \int_u^{\infty} \left[ c\varphi (x\textrm{/}c)+x\Phi (x\textrm{/c})\right]\varphi (x)dx - \varphi (u)\\
= &  \int_u^{\infty} \left[ c\varphi (x\textrm{/}c)+x\Phi (x\textrm{/c})\right]\varphi (x)dx - \int_u^{\infty} x\varphi (x)dx\\
=&  \int_u^{\infty} \left[ c\varphi (x\textrm{/}c)-x\overline{\Phi} (x\textrm{/c})\right]\varphi (x)dx \geq 0,
\end{array}
\end{displaymath}
since when $x\geq 0$,
$$\frac{\varphi(x)}{x}\geq \overline{\Phi}(x).$$
So the term in the direct method is always larger when $u\geq 0$. \medskip  \\
 Let us consider now the two terms corresponding to $\sigma_2(S)$:
 
  \begin{itemize}
\item $A_d =  u\varphi(u) +  \int_u^{\infty} \left[ \left(\frac{2(c^2+1)}{3}\right)^{3/2} \sqrt{\pi}\frac{ \varphi(x/c)}{c}\right]\varphi(x)dx = u \varphi(u)  + \overline{A}_d.$

\item $A_r =    \left[ c\varphi (u\textrm{/}c)+u\Phi (u\textrm{/c})\right]\varphi (u)   =   u\varphi(u) +\overline{A}_r.$
\end{itemize}
It is easy to show that , as $u\to +\infty$,
$$
 \overline{A}_d= \int _u ^{\infty} \varphi \bigg( \frac x c\bigg) \varphi(x) dx = (const)  \overline{\Phi} \bigg( u \sqrt{ \frac{1+ c^2}{ c^2}}\bigg) \simeq (const)  u^{-1} \varphi \bigg( u \sqrt{ \frac{1+ c^2}{ c^2}}\bigg).
 $$
 and that 
 $$
  \overline{A}_r \simeq  (const)  u^{-2} \varphi \bigg( u \sqrt{ \frac{1+ c^2}{ c^2}}\bigg).
 $$
This shows that for  $u$ sufficiently large $ A_r$  is smaller than $ A_d$. \\
 
 The numerical comparison  is performed in Figure 4  for six different situations.  It shows that  the record method is always better than the direct method.  EC method and record method are very close, but it is not possible  to identify  the better  among those two since  $\PP_E$  can be smaller than the true value.

\begin{figure}[!ht]
\begin{center}
\subfigure[$T=1, \; \rho''(0)=0.25$]{\includegraphics[scale=0.4]{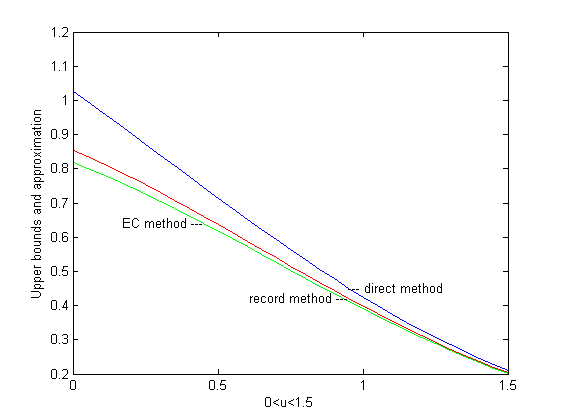}} 
\subfigure[$T=1, \; \rho''(0)=0.5$]{\includegraphics[scale=0.4]{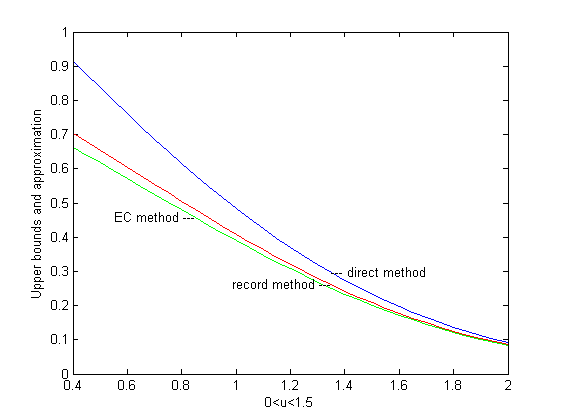}}
\subfigure[$T=2, \; \rho''(0)=0.5$]{\includegraphics[scale=0.4]{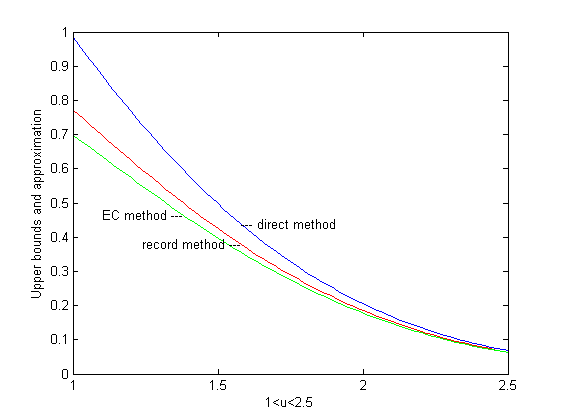}}
\subfigure[$T=2, \; \rho''(0)=1$]{\includegraphics[scale=0.4]{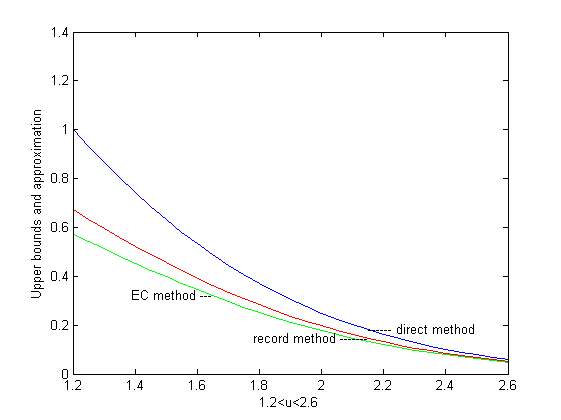}}
\subfigure[$T=4, \; \rho''(0)=2$]{\includegraphics[scale=0.4]{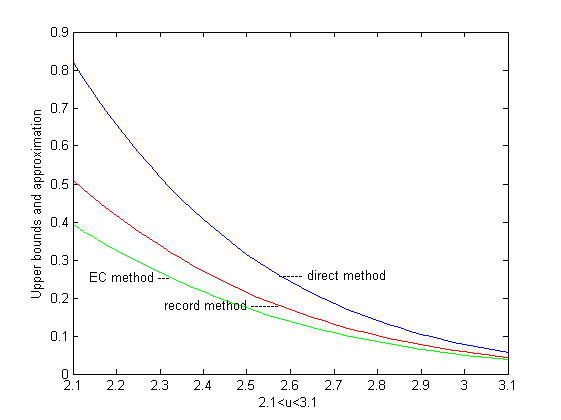}}
\subfigure[$T=0.25, \; \rho''(0)=0.5$]{\includegraphics[scale=0.4]{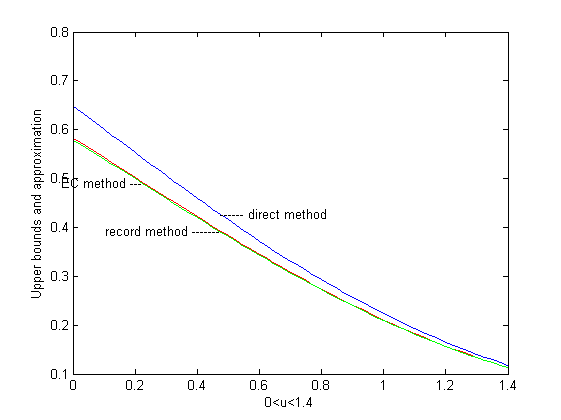}}
\caption {Comparison of the two bounds $\PP_R$ and $\PP_M$ and the approximation $\PP_E$ for several values of $\rho''(0)$ and $T$.}
\end{center}
\end{figure}

jean-marc.azais@math.univ-toulouse.fr\\
viet-hung.pham@math.univ-toulouse.fr
\end{document}